\UseRawInputEncoding

\documentclass[13pt]{amsart}

\usepackage{amsmath,amssymb,amsthm,amscd}
\usepackage{graphicx}
\usepackage{tikz}
\numberwithin{equation}{section}

\newtheorem{prop}{Proposition}[section]
\newtheorem{theo}[prop]{Theorem}
\newtheorem{lemm}[prop]{Lemma}

\newtheorem{rem}[prop]{Remark}

\newtheorem{defi}[prop]{Definition}
\newtheorem{conj}[prop]{Conjecture}

\def\begeq{\begin{equation}}
\def\endeq{\end{equation}}

\begin{document}

\title{K\"ahler-Einstein metrics and Ding functional on $\mathbb Q$-Fano group compactifications.}
\author{Yan Li$^{*}$ ZhenYe Li$^{\dag}$}

\address{$^{*}$School of Mathematics and Statistics, Beijing Institute of Technology, Beijing, 100081, China.}
\address{$^{\dag}$College of Mathematics and Physics, Beijing University of Chemical Technology,
Beijing, 100029, China.}
\email{liyan.kitai@yandex.ru\,\,lizhenye@pku.edu.cn}

\thanks {$^*$Partially supported by China Post-Doctoral Grant BX20180010 and Beijing Institute of Technology Research Fund Program for Young Scholars, No. 3170012222012.}
\subjclass[2000]{Primary: 53C25; Secondary:
32Q20, 58D25, 14L10}

\keywords{K\"ahler-Einstein metrics, $\mathbb Q$-Fano compactifications of Lie groups, moment polytopes, reduced Ding functional}

\begin{abstract}
Let $G$ be a complex, connect reductive Lie group which is the complexification of a compact Lie group $K$. Let $M$ be a $\mathbb Q$-Fano $G$-compactification. In this paper, we first prove the uniqueness of $K\times K$-invariant (singular) K\"ahler-Einstein metric. Then we show the existence of (singular) K\"ahler-Einstein metric implies properness of the reduced Ding functional. Finally, we show that the barycenter condition is also necessary of properness.
\end{abstract}
\maketitle

\section{Introduction}
Let $G$ be an $n$-dimensional connect, complex reductive group which is the complexification of a compact Lie group $K$, with complex structure $J_G$.
A projective normal variety $M$ is called a {\it (bi-equivariant) compactification of $G$} (or \emph{$G$-compactification} for simplicity) if it admits a holomorphic $G\times G$-action with an open and dense orbit isomorphic to $G$ as a $G\times G$-homogeneous space (cf. \cite{AB1, AB2, Timashev-Sbo}). If there is in addition a $G\times G$-linearized ample ($\mathbb Q$-Cartier) line bundle $L$ on $M$, then $(M, L)$ is called a {\it polarized compactification} of $G$ (cf. \cite[Section 2.1]{AK}). In particular, when $K^{-1}_M$ is an ample $\mathbb Q$-Cartier line bundle, we call $M$ a \emph{$\mathbb Q$-Fano $G$-compactification}. We refer the reader to \cite{Timashev-Sbo, AK, Del2, Del3}, etc. for further knowledge.

Fix a maximal complex torus $T^\mathbb C$ of $C$. Up to $G\times G$-equivariant isomorphisms, polarized $G$-compactifications are in one-one correspondence with its associated polytopes (see Section 2.1 below for detail) $P$ (cf. \cite[Theorem 2.4]{AK}), which lies in $J_g\mathfrak t^*$. Let $\Phi$ be the root system of $(G,T^\mathbb C)$ and $\Phi_+\subset\Phi$ be a chosen set of positive roots. We defined $P_+$ to be the intersection of $P$ with the positive Weyl chamber defined by $\Phi_+$. In the following we take $P$ to be the polytope of $(M,K^{-1}_M)$. In terms of a weighted barycenter $\mathbf{b}(2P_+)$ of twice the polytope $P_+$ (cf. \eqref{bar-def} below), Li-Tian-Zhu in \cite{LTZ2} proved the following criterion of existence of (singular) K\"ahler-Einstein metric on a $\mathbb Q$-Fano group compactification:
\begin{theo}\label{sg-bary-thm} Let $M$ be a $\mathbb Q$-Fano $G$-compactification whose associated polytope satisfies the fine condition.\footnote{By ``fine" we mean that each vertex of $P$ is the intersection of precisely $r=\text{rank}(G)$ facets.} Then $M$ admits a K\"ahler-Einstein metric if and only if
\begin{align}\label{0109}
\mathbf{b}(2P_+)\,\in\, 4\rho+\Xi,
\end{align}
where $2\rho=\sum_{\alpha\in\Phi_+}\alpha$
and
$\Xi$ is the relative interior of the cone generated by $\Phi_+$.
\end{theo}
Theorem \ref{sg-bary-thm} was first proved by Delcroix \cite{Del2} for smooth Fano compactifications (see also \cite{LZZ} for general polarized cases). In \cite{LTZ2}, Li-Tian-Zhu proved that \eqref{0109} implies properness of Ding functional (modulo group action) on the $\mathcal E^1_{K\times K}(M,K_M^{-1})$-space introduced by \cite{BBEGZ}. Hence prove the sufficiency of \eqref{0109} by using the variation method. In fact, the author proved that when \eqref{0109} holds, the reduced Ding functional $\mathcal D(\cdot)$ is proper on a set of convex functions $\mathcal E^1_{K\times K}(2P)$ on $2P$, which consists the reduction of $\mathcal E^1_{K\times K}(M,K_M^{-1})$-space (see Section 2.2 for detail).

On the other hand, for the direction of necessity, they showed that \eqref{0109} is necessary for existence of (singular) K\"ahler-Einstein metric by checking $K$-stability, using the argument of \cite[Section 3]{AK}. It is unknown whether \eqref{0109} is also necessary for the reduced Ding functional being proper.

In this paper, we will prove that the existence of (singular) K\"ahler-Einstein metric implies properness of reduced Ding functional. And this forces $\mathbf{b}(2P_+)$ to satisfy \eqref{0109}. Namely,
\begin{theo}\label{thm-1.2}
Let $M$ be a $\mathbb Q$-Fano $G$-compactification. If $M$ admits a (singular) K\"ahler-Einstein metric. Then
\begin{itemize}
\item [(1)] The reduced Ding functional $\mathcal D(\cdot)$ is proper on $\mathcal E^1_{K\times K}(2P)$, i.e.
\begin{align}\label{0108}
\mathcal D(u)\geq c_0\int_{2P_+}u\pi\,dy-C_0,~\forall u\in\mathcal E^1_{K\times K}(2P)
\end{align}
holds for suitable constants $c_0,C_0>0$.
\item [(2)] The barycenter $\mathbf {b}(2P_+)$ satisfies \eqref{0109}. Consequently, $M$ is K-stable.
\end{itemize}
\end{theo}
As a consequence, we get a direct proof of necessity of \eqref{0109}.

It is also worth mentioning that for a general smooth Fano manifold, Zhu \cite{Zhu-survey} proposed the following conjecture (see also \cite[Conjecture 7.7]{LZ-Sasaki}):
\begin{conj}\label{zhu-conj}\cite[Conjecture 4.9]{Zhu-survey}
Let $(M,g)$ be an $n$-dimensional K\"ahler-Einstein manifold with $\omega_g\in2\pi c_1(M)$. Denote by $\mathfrak K$ the maximal compact subgroup of the automorphism group $\text{Aut}(M)$. Suppose that $\omega_g$ is $\mathfrak K$-invariant. Then there are constants $c_0,C_0>0$ such that
\begin{align}\label{zhu-conj-eq}
\mathcal D(u)\geq c_0\inf_{\sigma\in Z(\text{Aut}(M))}I(\phi_\sigma)-C_0,~\forall \text{$\mathfrak K$-invariant K\"ahler potential $\phi$.}
\end{align}
Here $I(\cdot)$ is the Aubin's $I$-functional and $Z(\text{Aut}(M))$ denotes the center of $\text{Aut}(M)$.
\end{conj}
Very recently, Hisamoto \cite[Theorem 4.3]{Hisamoto} confirmed this conjecture. His approach uses deep results from \cite{Berman-Darvas-Lu} to check hypothesises of the properness pricinple in \cite[Section 3.2]{DR}. While by a reduction progress \cite[Lemma 4.10-4.11]{LZ}, it is direct to derive from Theorem \ref{thm-1.2} that an analogous of \eqref{zhu-conj-eq} holds if we replace $Z(\text{Aut}(M))$ by a possible larger group $Z(G)$, but on a larger space which consists of all $K\times K$-invariant K\"ahler potentials. It seems to be more nature and practical to consider $K\times K$ than $\mathfrak K$ on group compactifications, since in many cases, the full automorphism group $\text{Aut}(M)$, as well as $\mathfrak K$, is still unknown.

Our method is direct by studying the reduced Ding functional on a space of convex functions $\mathcal E^1_{K\times K}(2P)$. From existence to properness we want to use the argument from \cite{DR}. For this purpose, we need first to prove a uniqueness of $K\times K$-invariant (singular) K\"ahler-Einstein metrics (see Theorem \ref{uniqueness-T} below). Then we show the existence of (singular) K\"ahler-Einstein metric implies properness of the reduced Ding functional $\mathcal D(\cdot)$. Once this shown, \eqref{0109} then follows from an estimate of $\mathcal D(\cdot)$ along a special ray in $\mathcal E^1_{K\times K}(2P)$ (cf. Proposition \ref{stab-bary} (2) below).

The paper is organized as follows: in Section 2 we collect basic definitions and properties concerning polarized compactifications and reduced Ding functional on them. Section 3 is devoted to the uniqueness theorem for a general $\mathbb Q$-Fano variety, namely Theorem \ref{uniqueness-T}. In section 4 we prove Theorem \ref{thm-1.2}. The two Appendixes collect useful properties concerning structure of Aut$(M)$.

\subsection*{Notations} Now we fix the notations in the following sections except the Appendix. We denote by
\begin{itemize}
\item $K$-a connected, compact Lie group;
\item $G=K^\mathbb C$-the complexification of $K$, which is a complex, connected reductive Lie group;
\item $J_G$-the complex structure of $G$;
\item $T$-a fixed maximal torus of $K$ and $T^\mathbb C$ its complexification;
\item $\mathfrak a:=J_G\mathfrak t$-the non-compact part of $\mathfrak t^\mathbb C$;
\item $\Phi$-the root system with respect to $G$ and $T^\mathbb C$;
\item $W$-the Weyl group with respect to $G$ and $T^\mathbb C$;
\item $\Phi_+$-a chosen system of positive roots in $\Phi$;
\item $\text{Ad}_\sigma(\cdot):=\sigma(\cdot)\sigma^{-1}$-the conjugate of some subgroup or Lie algebra by some element $\sigma$.
\end{itemize}
\subsection*{Acknowledgement} The authors would like to sincerely thank Prof. M. Brion for telling them Proposition \ref{max-tor-unique} in \cite{Brion-letter}, which plays a crucial role in proving Theorem \ref{uniqueness-T}. They would also like to thank Professor X. Zhu for introducing us this problem, Prof. D. A. Timash\"ev, Prof. T. Delcroix, Prof. Chi Li and Prof. F. Wang for valuable comments.

\section{Preliminaries}
\subsection{Polarized compactifications and associated polytopes}
Let $(M,L)$ be a polarized compactification of $G$. It is known that the closure $Z$ of $T^\mathbb C$ in $M$, together with $L|_Z$ is a polarized toric variety. Indeed, $L|_Z$ is $WT^\mathbb C$-linearized. The polytope associated to $(M,K^{-1}_M)$ is defined as the associated polytope of $(Z,K_M^{-1}|_Z)$ (cf. \cite[Section 2.1]{AK} and \cite[Section 2.2]{Del2}). It is a strictly convex, $W$-invariant rational polytope in $\mathfrak a^*=J\mathfrak t^*$.
Denote by $\mathfrak a^*_+$ the corresponding positive Weyl chamber\footnote{When $G=T^\mathbb C$, we take $\mathfrak a^*_+=\mathfrak a^*.$}
$$\mathfrak a_+^*:=\{y\in\mathfrak a^*|\langle\alpha,y\rangle\geq0,~\forall\alpha\in\Phi_+\}.$$
Choose a $W$-invariant inner product $\langle\cdot,\cdot\rangle$ on $\mathfrak a^*$ which extends the Cartan-Killing form on the semisimple part $\mathfrak a^*_{ss}$ (cf. \cite[Introduction]{Del2}). Let $P_+$ be the positive part of $P$ defined by $P_+=P\cap\mathfrak a^*_+$.
The weighted barycenter of $2P_+$ is defined by
\begin{align}\label{bar-def}
\mathbf{b}(2P_+)\,=\,\frac{\int_{2P_+}y\pi(y) \,dy}{\int_{2P_+}\pi(y) \,dy},
\end{align}
where $\pi(y)=\prod_{\alpha\in\Phi_+}\langle\alpha,y\rangle^2.$

\subsection{The $\mathcal E^1$-space and reduced Ding functional}
\subsubsection{Singular K\"ahler-Einstein metric and the $\mathcal E^1$-space}
For a $\mathbb Q$-Fano variety $M$, by Kodaira's embedding Theorem, there is an integer $\ell > 0$ such that we can embed $M$ into a projective space $\mathbb {CP}^N$ by a basis of $H^0(M, K_M^{-\ell})$. Then we have a metric
$\omega_0\,=\,\frac{1}{\ell}\,\omega_\text{FS}|_{M}\,\in \,2\pi c_1(M),$
where $\omega_\text{FS}$ is the Fubini-Study metric of $\mathbb {CP}^N$.
Moreover, there is a Ricci potential $h_0$ of $\omega_0$ such that
$$\text{Ric}(\omega_0)-\omega_0\,=\,\sqrt{-1}\partial\bar\partial h_0, \text{ on the regular part } M_\text{reg}.$$
In case that $M$ has only klt-singularities, $e^{h_0}$ is $L^p$-integrate for some $p>1$ (cf. \cite{DT, BBEGZ}).

For a general (possibly unbounded) K\"ahler potential $\varphi$ and
$$\omega_\varphi:=\omega_0+\sqrt{-1}\partial\bar\partial\varphi,$$
define its complex Monge-Amp\`ere measure $\omega_{\varphi}^n$ by
$$\omega_{\varphi}^n\,=\,\lim_{j\to \infty}\,\omega_{\varphi_j}^n,$$
where $\varphi_j\,=\,{\rm max}\{\varphi,-j\}$.
According to \cite{BBEGZ}, we say that $\varphi$ (or $\omega_{\varphi}^n$)  has  full Monge-Amp\'ere (MA) mass if
$$\int_M \omega_{\varphi}^n\,=\,\int_M \omega_0^n.$$
The MA-measure $\omega_{\varphi}^n$ with full MA-mass has no mass on the pluripolar set of $\varphi$ in $M$. Thus we only need to consider the
measure on the regular part $M_\text{reg}$.  

\begin{defi}\label{sg-singular-ke} We call $\omega_{\varphi}$ a (singular) K\"ahler-Einstein metric on $M$ with full MA-mass if $\omega_{\varphi}^n$ has full MA-mass and $\varphi$ satisfies the following complex Monge-Amp\'ere equation,
\begin{align}\label{sg-singular-ke-equation}
\omega_{\varphi}^n\,=\,e^{h_0-\varphi}\omega_0^n.
\end{align}
\end{defi}

As in the smooth case, there is a well-known Euler-Langrange functional for K\"ahler potentials associated to \eqref{sg-singular-ke-equation},  often referred as the Ding functional or F-functional, defined by (cf. \cite{Di88})
\begin{align}\label{ding-functional}
 F(\phi)\,=\,-\frac1{(n+1)V}\sum_{k=0}^n\,\int_M\phi\omega_\phi^k\wedge\omega_0^{n-k}-\log\left(\frac1V\int_Me^{h_0-\phi}\omega_0^n\right).
\end{align}
On a $\mathbb Q$-Fano manifold with klt-singularities, Berman-Boucksom-Eyssidieux-Guedj-Zeriahi \cite{BBEGZ} showed that $F(\cdot)$ can be defined on the space $\mathcal E^1(M,-K_M)$ given by
\begin{align}\mathcal E^1(M,-K_M)\,=\,\{\phi| ~&\phi \text{ has full MA mass and }  \notag\\
& \sup_M\phi =0,~ I(\phi)=\int_M-\phi\omega_\phi^n<\infty\}.\notag
\end{align}
They also showed that $\mathcal E^1(M,-K_M)$ is compact in certain weak topology. By a result of Davas \cite{DT},
$\mathcal E^1(M,-K_M)$ is in fact compact in the topology of $L^1$-distance. This provides a variational approach to \eqref{sg-singular-ke-equation}.

It has been shown in \cite{BBEGZ} that if $\varphi$ is a solution of  \eqref{sg-singular-ke-equation}, then it is $C^\infty$ on $M_{\rm reg}$. Thus $\omega_{\varphi}$ satisfies the usual K\"ahler-Einstein equation $\text{Ric}(\omega_{\varphi})=\omega_{\varphi}$ on $M_\text{reg}$.

\subsubsection{$\mathcal E^1$-space and reduced Ding functional on $G$-compactifications}
By the standard $KAK$-decomposition \cite[Section 3.5.3]{Zhelobenko-Shtern}, there is a bijection between $K\times K$-invariant functions $\Psi$ on $G$ and $W$-invariant functions  $\psi$ on
$\mathfrak a$ which is given  by
$$\Psi( \exp(\cdot))\,=\,\psi(\cdot):~{\mathfrak a}\to\mathbb R.$$
Clearly, when a $W$-invariant $\psi$ is given, $\Psi$ is well-defined, and vice versa. From now on, for simplicity, we will not distinguish $\psi$ and $\Psi$,  and we will call $\Psi$ convex on $G$ if  $\psi$  is convex on ${\mathfrak a}$.

Such a correspondence can be extended to $K\times K$-invariant quantities on whole $M$ (see \cite[Section 3.4]{Timashev-lecture-en}). For example, let $\omega\in2\pi c_1(L)$ be a $K\times K$-invariant K\"ahler metric of $M$. Then $\omega|_Z$ is a toric K\"ahler metric in $2\pi c_1(L|_Z)$ and there is a $W$-invariant, strictly convex function $\psi$ on $\mathfrak a$ such that (cf. \cite{AL})
\begin{align*}
\omega=\sqrt{-1}\partial\bar\partial \psi, ~{\rm on}~ G.
\end{align*}
Conversely, any $W$-invariant toric K\"ahler metric $\omega=\sqrt{-1}\partial\bar\partial \psi\in2\pi c_1(L|_Z)$ on $Z$ extends to a $K\times K$-invariant K\"ahler metric $\omega\in2\pi c_1(L)$ on $M$.

Denotr by $O$ the origin. In the following, we take $L=K^{-1}_M$ and fix a background metric $\omega_0=\sqrt{-1}\partial\bar\partial\psi_0$ such that
\begin{align}\label{normalization}
\inf_{\mathfrak a} \psi_0\,=\, \psi_0(O)\,=\,0,
\end{align}
We restrict ourself to $\mathcal E_{K\times K}^1(M,K^{-1}_M)$, the space of $K\times K$-invariant functions $\phi$ in $\mathcal E^1(M,K^{-1}_M)$, such that $\psi_\phi=\psi_0+\phi$ is normalized as \eqref{normalization}. For any such $\phi$, consider the Legendre function $u_\phi$ of $\psi_\phi$. As in \cite{Coman-Guedj-Sahin-Zeriahi}, it is showed in \cite[Section 4]{LTZ2} that
\begin{theo}\label{E1-legendre}
 A  K\"ahler potential $\phi\in\mathcal E^1_{K\times K}(M,-K_M)$  if and only if the Legendre function $u_\phi$ of $\psi_\phi$ lies in
 \begin{align}\mathcal E^1_{K\times K}(2P)=\{u| &\text{$u$ is convex, $W$-invariant on $2P$ which satisfies}
\notag\\
&\inf_{2P}u=u(O)=0\text{ and }\int_{2P_+} u \pi\,dy<+\infty\}.\notag
\end{align}
Hence the Legendre transformation gives a bijection between $\mathcal E^1_{K\times K}(M,-K_M)$ and $\mathcal E^1_{K\times K}(2P)$. As a consequence, $u_\phi$ is locally bounded in  Int$(2P)$ if $\phi\in\mathcal E^1_{K\times K}(M,-K_M)$.
\end{theo}

For any $u\in{\mathcal  E}^1_{K\times K}(2P)$, its Legendre function
$$\psi_u(x)\,=\,\sup_{y\in 2P}\{\langle x, y\rangle-u(y)\}\,\le\, v_{2P}(x)$$
corresponds to  a $K\times K$-invariant weak K\"ahler potential $\phi_u=\psi_u-\psi_0$ which belongs to $\mathcal E^1_{K\times K}(M,-K_M)$. 
As we know,  $e^{-\phi_u}\in L^p(\omega_0)$ for any $p\ge 0$. Thus $$\int_{\mathfrak a_+}e^{-\psi_u}\mathbf J(x)dx$$ is well-defined.

Define the following  functional   on $\mathcal E^1_{K\times K}(2P)$ by
  $$\mathcal D(u)\,=\,\mathcal L(u)+\mathcal F(u),$$
where
\begin{align}\label{L(u)}
\mathcal L(u)\,=\,\frac1V\int_{2P_+}u\pi\,dy-u(4\rho)
\end{align}
and
\begin{align}\label{F(u)-0}
\mathcal F(u)\,=\,-\log\left(\int_{\mathfrak a_+}e^{-\psi_u}\mathbf J(x)dx\right)+u(4\rho).
\end{align}
It is easy to see that on  a $\mathbb Q$-Fano  compactification  of $G$,
$$\mathcal L(u_\phi)+u_\phi(4\rho)\,=\,-\frac1{(n+1)V}\sum_{k=0}^n\int_M\phi\omega_\phi^k\wedge\omega_0^{n-k}$$
and $\mathcal D(u_\phi)$ is just the Ding functional $F(\phi)$ for $\phi\in\mathcal E_{K\times K}(M,K^{-1}_M)$ \eqref{ding-functional}. Hence in the following we call $\mathcal D(\cdot)$ the reduced Ding functional.

It is showed in \cite[Section 6]{LTZ2} that under the condition \eqref{0109}, the minimizer of the reduced Ding functional on the space $\mathcal E^1_{K\times K}(M,K_M^{-1})$ exists and is a solution of \eqref{sg-singular-ke-equation}.

\section{Uniqueness of (singular) K\"ahler-Einstein metrics}
Let $M$ be any compact $\mathbb Q$-Fano variety. It is well-known that when $M$ admits a (singular) K\"ahler-Einstein metrics, the neutral component ${\rm Aut}^0(M)$ of its automorphism group ${\rm Aut}(M)$ is reductive, and is the complexification of the isometry group of this (singular) K\"ahler-Einstein metrics \cite[Theorem 5.2]{BBEGZ}. Furthermore, the  (singular) K\"ahler-Einstein metrics is unique up to an ${\rm Aut}^0(M)$-action \cite[Theorem 5.1]{BBEGZ}. In this section, we will further discuss the uniqueness of $K\times K$-invariant (singular) K\"ahler-Einstein metrics on a $\mathbb Q$-Fano $G$-compactification and prove that any such two metrics can only be different from a $Z(G)$-action.

\begin{theo}\label{uniqueness-T}
Let $M$ be a $\mathbb Q$-Fano $G$-compactification. Then the $K\times K$-invariant (singular) K\"ahler-Einstein metric, if exists, is unique up to a $Z(G)$-action.
\end{theo}

\begin{proof}
Denote by $\mathfrak G={\rm Aut}^0(M)$ and fix $\mathfrak K$ a maximal compact subgroup of $M$. Let $\omega_1,\omega_2$ be two $K\times K$-invariant (singular) K\"ahler-Einstein metrics and denote by $\mathfrak K_1,\mathfrak K_2$ their groups of isometry, respectively.
Then $\mathfrak G=\mathfrak K_1^\mathbb C=\mathfrak K_2^\mathbb C$. Hence there is a $\sigma\in\mathfrak G$ such that $\text{Ad}_\sigma\mathfrak K_1=\mathfrak K_2$.

Since $\mathfrak G$ is reductive, by Proposition \ref{max-tor-unique}, $T\times T(\subset K\times K)$ extends to a unique maximal torus $\mathfrak T$ of $\mathfrak G$. Hence $\mathfrak T\subset\mathfrak K_1\cap\mathfrak K_2$ and up to replacing $\sigma$ by $\tilde{\sigma}\cdot\sigma$ for some $\tilde\sigma\in\mathfrak K_2$, we have $\text{Ad}_\sigma\mathfrak T=\mathfrak T$. Thus both $\text{Ad}_{\sigma^{-1}}\mathfrak T$ and $\mathfrak T$ are contained in $\mathfrak K_1$.

Since $\mathfrak T$ is maximal, there are $\sigma'\in N_{\mathfrak K_1}(\mathfrak T),\sigma''\in\mathfrak K_1$ and a $t\in\mathfrak T^\mathbb C$ such that $\sigma''\cdot\sigma^{-1}=\sigma'\cdot t$. We have $t=(\sigma'^{-1}\cdot\sigma'')\cdot\sigma^{-1}$ Hence $\text{Ad}_{t^{-1}}\mathfrak K_1=\mathfrak K_2$.

On the other hand, as both $\mathfrak K_1$ and $\mathfrak K_2=\text{Ad}_{t^{-1}}\mathfrak K_1$ contains $K\times K$, we see that
$$K\times K,\text{Ad}_{t^{-1}}(K\times K)\subset\mathfrak K_1.$$
By Lemma \ref{centralizer}, there is a $\tilde\sigma\in\mathfrak T$ such that $t':=t^{-1} \cdot\tilde\sigma\in C_{K\times K}(\mathfrak G)$, and
\begin{align*}
\omega_2=\sigma^*\omega_1=&(t^{-1}\cdot(\sigma'^{-1}\cdot\sigma''))^*\omega_1\\
=&(t^{-1})^*\omega_1=t'^*(\tilde\sigma^{-1})^*{\omega_1}=t'^*{\omega_1}.
\end{align*}
Recall that $G\times G=(K\times K)^\mathbb C$. Hence $t'\in C_{G\times G}(\mathfrak G)$. By \cite[Proposition 1.8]{Timashev-book}, $t'$ can be realized by a $Z(G)$-action.

\end{proof}

\section{Properness of Modified Ding functional}
In this section, we will prove that when a $\mathbb Q$-Fano $G$-compactification $M$ admits a (singular) K\"ahler-Einstein metric, the reduced Ding functional defined in \cite{LTZ2} will be proper. Moreover, this implies \eqref{0109} and consequently by \cite{Del3}, $M$ is K-stable.

\begin{prop}\label{stab-bary}
Suppose that the barycenter $\mathbf{b}(2P_+)\in \mathfrak a_{ss}$. Then:
\begin{itemize}
\item[(1)] If $\mathcal D(\cdot)$ is bounded from below on $\mathcal E^1_{K\times K}(2P)$. Then
\begin{align}\label{0107}
\mathbf{b}(2P_+)-4\rho\in\overline{\Xi}.
\end{align}
\item [(2)] If \eqref{0108} holds for some uniform constants $c_0,C_0>0$. Then we have \eqref{0109}.

\end{itemize}
\end{prop}

\begin{proof}
For item (1). Suppose that \eqref{0107} is not true. With out loss of generality we may assume
\begin{align}\label{0110}\mathbf{b}(2P_+)-4\rho=\sum_{i=1}^rc_i\alpha_i\end{align}
with $c_1<0$. Here $\Phi_{+,s}=\{\alpha_1,...,\alpha_r\}$ are the simple roots. Let $\{\varpi_i\}_{i=1}^r$ be the corresponding fundamental weights such that
$$\langle\alpha_i,\varpi_j\rangle=\frac12|\alpha_i|^2\delta_{ij}.$$
For $\lambda>0$, set
$$u_\lambda(y)=\left\{\begin{aligned}&\lambda\max\{\langle w(\varpi_1),y\rangle|w\in W\},&y\in 2P,\\
&+\infty,&\text{otherwise.}\end{aligned}\right.$$
Then by \eqref{L(u)},
\begin{align*}
\mathcal L(u_\lambda)=\lambda\cdot \frac12c_1|\alpha_1|^2.
\end{align*}

Let $\psi_\lambda$ be the Legendre function of $u_\lambda$ and $$\widetilde{\psi}_\lambda=\psi_\lambda-4\rho(x).$$
Then by \eqref{F(u)-0} and \cite[Lemma 4.8]{LZ},
\begin{align}\label{F(u)}
\mathcal F(u_\lambda)=-\log\int_{\mathfrak a_+}e^{-(\widetilde{\psi}_\lambda-\inf_{\mathfrak a_+}\widetilde{\psi}_\lambda)}\prod_{\alpha\in\Phi_+}(\frac{1-e^{-2\alpha(x)}}2)^2dx.
\end{align}

We want to compute out $\psi_\lambda$. Divide $\mathfrak a^*$ into cones $\sigma_1,...,\sigma_{s_0}$ so that each $\sigma_i$ is a linear domain of the function $u_\lambda$, that is
$$u_\lambda|_{\sigma_i}=\lambda\langle w_i(\varpi_i),y\rangle\text{ for some $w_i\in W$.}$$
Clearly there is a unique one that contains $\mathfrak a^*_+$ and we denote it by $\sigma_+$.

For a convex domain $\Omega\subseteq\mathbb R^n$, denote by $v_\Omega$ its support function. By definition,
\begin{align*}
\psi_\lambda(x)&=\sup_{y}\{\langle x,y\rangle-u_\lambda(y)\}\\
&=\max_{i=1,...,s_0}\sup_{y\in 2P\cap\sigma_i}\{\langle x,y\rangle-\lambda\langle w_i(\varpi_i),y\rangle\}\\
&=\max_{i=1,...,s_0}v_{2P\cap\sigma_i}(x-\lambda w_i(\varpi_1)).
\end{align*}
Since $4\rho\in\text{Int}(2P\cap\sigma_+)$, it is direct to check that
\begin{align*}
v_{2P\cap\sigma_+}(x-\lambda\varpi)-4\rho(x-\lambda\varpi)-4\lambda\rho(\varpi)\geq0.
\end{align*}
It then follows
\begin{align*}
\tilde\psi_\lambda(x)\geq\tilde\psi(\lambda\varpi_1)=-4\lambda\rho(\varpi_1).
\end{align*}
Consequently,
\begin{align*}
\tilde\psi_\lambda(x)-\inf_{\mathfrak a_+}\tilde\psi_\lambda\leq v_{2P}(x-\lambda\varpi_1)-4\rho(x-\lambda\varpi_1)=:\tilde v_{2P}(x-\lambda\varpi_1).
\end{align*}

Note that $4\rho\in\text{Int}(2P)$, thus $\tilde v_{2P}(x-\lambda\varpi_1)$ is proper on $\mathfrak a$. Fix a $\delta_1>0$, there is a $\lambda_1(\delta_1)$ and a convex domain $\Omega(\delta_1)$ such that for any $\lambda\geq\lambda_1(\delta_1)$, we have
\begin{align}\label{0102+}
\tilde\psi_\lambda(x)-\inf_{\mathfrak a_+}\tilde\psi_\lambda\leq\tilde v_{2P}(x-\lambda\varpi_1)\leq\delta_1,~\forall x\in(\lambda\varpi_1+\Omega(\delta_1))\subset\mathfrak a_+.
\end{align}
Here $\Omega(\delta_1)$ can be taken as $$\Omega(\delta_1)=\{\tilde v_{2P}(x)= v_{2P}(x)-2\sum_{\alpha\in\Phi_+}(x)\leq\delta_1\}\cap\left(\cap_{i\geq2}\{\alpha_i(x)\geq0\}\right).$$

On the other hand, for any $\delta_0>0$, as $\lambda\to+\infty$, we have
\begin{align*}
1-e^{-2\lambda\alpha(x)}\geq{\delta_0}
\end{align*}
whenever $\alpha(x)\geq-\frac1{2\lambda}\log(1-{\delta_0}).$

Thus, we can choose a sufficiently small $\delta_0$ which does not depend on $\lambda$, so that the domain
$$\Omega_0:=(\lambda\varpi_1+\Omega(\delta_1))\cap\left(\cap_{\alpha\in\Phi_+}\{1-e^{-2\lambda\alpha(x)}\geq{\delta_0}\}\right)$$
satisfies
\begin{align}\label{0104}
\text{Vol}(\Omega_0)\geq V_0,~\forall \lambda\geq\lambda_1(\delta_1),
\end{align}
for some constant $V_0>0$.
By \eqref{F(u)}, \eqref{0102+} and \eqref{0104}, we have
\begin{align*}
\mathcal F(u_\lambda)&\leq-\log\int_{(\lambda\varpi_1+\Omega_0)}e^{-\delta_1}({\delta_0})^{n-r}dx=-\log(V_0e^{-\delta_0}\delta_0^{n-r}),~\lambda\to+\infty.
\end{align*}

Combining the above estimate with \eqref{L(u)}, we see that
\begin{align}\label{0106}
\mathcal D(u_\lambda)&\leq\frac12|\alpha_1|^2c_1\lambda+C_0,~\lambda\to+\infty,
\end{align}
where $C_0=-\log(V_0e^{-\delta_0}\delta_0^{n-r})$. The right-hand side goes to $-\infty$ since $c_1<0$. A contradiction.

The proof of item (2) goes in a similar way. Suppose that \eqref{0109} is not true, we may assume in \eqref{0110} that $c_1\leq0$. It then follows that \eqref{0106} contradicts with \eqref{0108}.
\end{proof}

\begin{rem}\label{rmk-4.2}
The condition $\mathbf {b}(2P_+)\in\mathfrak a_{ss}$ is equivalent to that the Futaki invariant of the compactification vanishes. This is because if $u_0\in\mathcal E^1_{K\times K}(2P)$ whose Legendre function corresponds to a K\"ahler metric $\omega_0\in2\pi c_1(M)$ and $\theta\in\mathfrak a\cap\mathfrak {z(g)}$, then $\omega_{\phi_\lambda}:=(e^{\lambda\xi})^*\omega_0$ is a family of $K\times K$-invariant metrics. Following the argument of \cite[Lemma 4.1]{LZ}, it is easy to see that
$$Fut(\xi)=\left.\frac{d}{d\lambda}\right|_{\lambda=0}F(\phi_\lambda)=\mathbf{b}(2P_+)(\xi).$$




\end{rem}

\begin{rem}
It is showed in \cite[Section 6]{LTZ2} that $\mathcal D(\cdot)$ is convex along any linear path in $\mathcal E_{K\times K}^1(2P)$.
Thus when $M$ admits a (singular) K\"ahler-Einstein metric $\omega_0$ (whose symplectic potential is $u_0$), then $$\mathcal D(u)\geq \mathcal D(u_0).$$
From Proposition \ref{stab-bary} (1) and \cite{Del3}, we conclude that $M$ is K-semistable. This is proved for a general $\mathbb Q$-Fano variety by Berman \cite{Berman-inve}.
\end{rem}

Now we prove  Theorem \ref{thm-1.2}. We first show the existence of $K\times K$-invariant geodesics in $\mathcal E^1(M,-K_M)$. This is a special case of the following lemma:
\begin{lemm}\label{geodesic-invariant}
Let $\hat K\subset \text{Aut}(M)$ be a compact group. Suppose that $\phi_0,\phi_1\in \mathcal E^1_{\hat K}(M,K_M^{-1})$ are two $\hat K$-invariant K\"ahler potentials and $\{\phi(x,t)\}_{t\in\{\mathbb C|\text{Re}(z)\in[0,1]\}}\subset\mathcal E^1(M,K_M^{-1})$ is a geodesic connecting them. Then for every $t$, the function $\phi(\cdot,t):M\to\mathbb C$ is $\hat K$-invariant.
\end{lemm}
\begin{proof}
By definition \cite[Section 4.2]{BBEGZ},
\begin{align*}
\phi(x,t)=\sup_{\tilde \phi}\{\tilde\phi(x,t)|\tilde\phi|_{\text{Re}(t)=0}(z,t)\leq\phi_0(z),\tilde\phi|_{\text{Re}(t)=1}(z,t)\leq\phi_1(z),~\forall z\in M\},
\end{align*}
where $\tilde \phi$ is a family of continuous $\omega_0$-psh functions. Thus for any $\hat k \in \hat K$
\begin{align*}
\phi(\hat kx,t)&=\sup_{\tilde \phi}\{\tilde\phi(\hat kx,t)|\tilde\phi|_{\text{Re}(t)=0}(z,t)\leq\phi_0(z),\tilde\phi|_{\text{Re}(t)=1}(z,t)\leq\phi_1(z)\}\\
&=\sup_{\tilde \phi}\{\tilde\phi(x,t)|\tilde\phi|_{\text{Re}(t)=0}(\hat k^{-1}z,t)\leq\phi_0(z),\tilde\phi|_{\text{Re}(t)=1}(\hat k^{-1}z,t)\leq\phi_1(z)\}\\
&=\sup_{\tilde \phi}\{\tilde\phi(x,t)|\tilde\phi|_{\text{Re}(t)=0}(\hat k^{-1}z,t)\leq\phi_0(\hat k^{-1}z),\tilde\phi|_{\text{Re}(t)=1}(\hat k^{-1}z,t)\leq\phi_1(\hat k^{-1}z)\}\\
&=\phi(x,t).
\end{align*}
Hence $$\{\phi(x,t)\}_{t\in\{\mathbb C|\text{Re}(z)\in[0,1]\}}\subset\mathcal E^1_{\hat K}(M,K_M^{-1})$$ as desired.
\end{proof}

Also, we have
\begin{lemm}
Let $\phi_1,\phi_2\in\mathcal E_{K\times K}^1(M,-K_M)$ and $u_1,u_2$ be the Legendre functions of $\psi_{\phi_1},\psi_{\phi_2}$, respectively. Then
\begin{align}\label{geodesic-eq}
u_t=(1-t)u_0+tu_1
\end{align}
is a geodesic joining $u_1,u_2$ and the distance
$$d_{\mathcal E^1}(\phi_1,\phi_2):=\int_{2P_+}|u_1-u_2|\pi dy.$$
\end{lemm}
\begin{proof}From Lemma \ref{geodesic-invariant}, we can take a $K\times K$-invariant geodesic $\phi_t$ such that
$$d_{\mathcal E^1}(\phi_1,\phi_2)=\int_0^1\int_M|\dot\phi_t|\omega_{\phi_t}^n\wedge dt.$$
By \cite[Lemma 4.5]{LTZ2}, we have
\begin{align*}
d_{\mathcal E^1}(\phi_1,\phi_2)&=\int_0^1\int_{\mathfrak a_+}|\dot\phi_t|\text{MA}_{\mathbb R;\pi}(\phi_t)\wedge dt\\
&=\int_0^1\int_{2P_+}|\dot u_t|\pi dy\wedge dt\\
&\geq\int_{2P_+}|\int_0^1\dot u_t dt|\pi dy=\int_{2P_+}|u_1-u_2|\pi dy.
\end{align*}
On the other hand, for the path \eqref{geodesic-eq}, it holds
$$\int_{2P_+}|u_1-u_2|\pi dy=\int_0^1\int_{2P_+}|\dot u_t|\pi dy\wedge dt\geq d_{\mathcal E^1}(\phi_1,\phi_2).$$
We conclude the lemma.
\end{proof}

\begin{proof}[Proof of Theorem \ref{thm-1.2}]
We first show (1). We will use an argument of \cite[Section 3]{DR}. Suppose that the Legendre function $\psi_0$ of some $u_0\in\mathcal E_{K\times K}^1(2P)$ defines a (singular) K\"ahler-Einstein metric $\omega_0=\sqrt{-1}\partial\bar\partial\psi_0$ on $M$. By \cite[Lemma 6.4]{LTZ2}, we see that $\mathcal D(\cdot)$ is convex along the path \eqref{geodesic-eq} for any $u_0,u_1$ in $\mathcal E^1_{K\times K}(2P_+)$. On the other hand, by \cite[Claim 6.7]{LTZ2}, $u_0$ is a critical point of $\mathcal D(\cdot)$ on $\mathcal E^1_{K\times K}(2P_+)$. Thus
\begin{align}\label{0301}
\mathcal D(u)\geq\mathcal D(u_0),~\forall u\in\mathcal E^1_{K\times K}(2P_+).
\end{align}

Consider the ratio
\begin{align}\label{+110302}
C:=\inf\{\frac{\mathcal D(u)-\mathcal D(u_0)}{d(u_0,u)}|u\in\mathcal E_{K\times K}^1(2P),~d(u_0,u)\geq1\},
\end{align}
where $$d(u_0,u)=\int_{2P_+}|u-u_0|\pi dy.$$
It suffice to prove that the constant $C$ in \eqref{+110302} is positive.

Suppose that $C=0$. Then we can find a sequence $\{u_k\}_{k\in\mathbb N_+}$ such that
\begin{align}\label{assump-contra}
\lim_{k\to+\infty}\frac{\mathcal D(u_k)-\mathcal D(u_0)}{d(u_0,u_k)}=0.
\end{align}
For each $k\in\mathbb N_+$, consider the path
\begin{align}\label{0305}
u_k(t)=(1-t)u_0+tu_k,~t\geq0.
\end{align}
Since for each $k$, $\mathcal D(u_k(t))$ is convex on $[0,+\infty)$. By \eqref{0301} we have
\begin{align}\label{0302}
0\leq\frac{\mathcal D(u_k(t))-\mathcal D(u_0)}{t}\leq\frac{\mathcal D(u_k)-\mathcal D(u_0)}{d(u_0,u_k)}.
\end{align}
Take $\hat u_k=u_k(1)$. Then
\begin{align*}
{d(u_0,\hat u_k)}=1,
\end{align*}
and by \eqref{0302},
\begin{align*}
0\leq{\mathcal D(\hat u_k)-\mathcal D(u_0)}\leq\frac{\mathcal D(u_k)-\mathcal D(u_0)}{d(u_0,u_k)}.
\end{align*}
By \cite[Proposition 6.5]{LTZ2}, $\mathcal D(\cdot)$ is semi-continuous on $\mathcal E_{K\times K}^1(2P)$. Thus by \eqref{0305}, up to a subsequence,
$$\hat u_k\to\hat u_0,~k\to+\infty$$
for some $\hat u_0\in\mathcal E_{K\times K}^1(2P)$ which also gives a (singular) K\"ahler-Einstein metric $\hat\omega_0$ on $M$. Also,
\begin{align}\label{dist-two-metrics}
d(\hat u_0,u_0)=d_{\mathcal E^1}(\hat\psi_0,\psi_0)=1.
\end{align}

On the other hand, by Theorem \ref{uniqueness-T}, there is some $\sigma\in Z(G)$ such that
$\hat\omega_0=\sigma^*\omega_0$. Hence the corresponding Legendre function $\hat u_0$ of $\hat \psi_0=\sigma^*\psi_0$ is $$\hat u_0=u_0-Z_\sigma^iy_i,$$
where $\sigma=e^{Z_\sigma}$ for some $Z_\sigma=(Z_\sigma^1,...,Z_\sigma^n)\in\mathfrak {z(g)}$. Since $u_0\in\mathcal E^1_{K\times K}(2P)$ is normalized at $O\in2P$, the assumption that $\hat u_0\in\mathcal E^1_{K\times K}(2P)$ implies $$Z_\sigma=0.$$ Hence $\hat u_0=u_0$, which contradicts to \eqref{dist-two-metrics}. We see that \eqref{assump-contra} can not hold and Theorem \ref{thm-1.2} (1) is true.

Since $M$ admits a (singular) K\"ahler-Einstein metric, it must have vanishing Futaki invariant (see Remark \ref{rmk-4.2} above) and hence $\mathbf{b}(2P_+)\in\mathfrak a_{ss}$.
The relation \eqref{0109} then follows from Proposition \ref{stab-bary} (2). The stability result follows from \eqref{0109} and a result in \cite{Del3}.

\end{proof}


\section{Appendix 1: An algebraic lemma}

In this appendix, we prove an elementary algebraic lemma for reductive groups.

\begin{lemm}\label{centralizer}
Suppose that $K$ is a compact, connected Lie group and $G=K^\mathbb C$ be its complexification. Let $T$ be a maximal torus of $K$ and $K'$ be a Lie subgroup of $K$. Suppose that there is a $t\in T^\mathbb C$ such that $\text{Ad}_tK'\subset K$. Then $t\in (C_G(K')\cap T^\mathbb C)\cdot T.$
\end{lemm}

\begin{proof}

Recall the Cartan decomposition
\begin{align}\label{crtan-dec}
\mathfrak g=\mathfrak t^\mathbb C\oplus\left(\otimes_{\alpha\in\Phi_+}(V_\alpha\oplus V_{-\alpha})\right),
\end{align}
where $\Phi_+$ is a set of positive roots corresponding to $(G,T^\mathbb C)$. Furthermore, for each root $\beta\in\Phi$, there is a $X_\beta$ such that
$$V_\beta=\mathbb C\cdot X_\beta.$$
Denote by $J$ the complex structure of $G$, we have
$$\mathfrak k=\mathfrak t\oplus(\oplus_{\alpha\in\Phi_+}\mathfrak k_\alpha),$$
where $$\mathfrak t=\text{Span}_{\mathbb R}\{E_1,...,E_r\}$$ and
$$\mathfrak k_\alpha=\text{Span}_{\mathbb R}\{(X_\alpha-X_{-\alpha}),J(X_\alpha+X_{-\alpha})\}.$$

Let $$t=e^{T_1+JT_2},$$ where $T_i\in\mathfrak t$. Then for any root $\beta\in\Phi$,
\begin{align*}
\text{Ad}_t\left(\begin{aligned}& X_\beta\\&JX_{\beta}\end{aligned}\right)=\left(\begin{aligned}& e^{-\beta(T_2)}\cos\beta(T_1)&e^{-\beta(T_2)}\sin\beta(T_1)\\-&e^{-\beta(T_2)}\sin\beta(T_1)&e^{-\beta(T_2)}\cos\beta(T_1)\end{aligned}\right)\left(\begin{aligned}& X_\beta,\\&JX_{\beta}\end{aligned}\right).
\end{align*}
Thus for any root $\beta\in\Phi_+$,
\begin{align*}
\text{Ad}_t(X_\beta-X_{-\beta})=&e^{-\beta(T_2)}[\cos\beta(T_1)(X_\beta-X_{-\beta})+\sin\beta(T_1)J(X_\beta+X_{-\beta})]\notag\\
&-2\sinh\beta(T_2)(\cos\beta(T_1)X_{-\beta}-\sin\beta(T_1)JX_{-\beta}),
\end{align*}
and
\begin{align*}
\text{Ad}_t J(X_\beta+X_{-\beta})=&e^{\beta(T_2)}[\sin\beta(T_1)(-X_\beta+X_{-\beta})+\cos\beta(T_1)J(X_\beta+X_{-\beta})]\notag\\
&+2\sinh\beta(T_2)(\sin\beta(T_1)X_{\beta}-\cos\beta(T_1)JX_{\beta}).
\end{align*}

Hence we get for any
$$X=\sum_{j=i}^ra_jE_j+\sum_{\alpha\in\Phi_+}(c_\alpha(X_\alpha-X_{-\alpha})+d_\alpha J(X_\alpha+X_{-\alpha})),$$
where $a_i,c_\alpha,d_\beta\in\mathbb R$, we get
\begin{align*}
\text{Ad}_tX\equiv2\sinh\beta(T_2)&\left[\sum_{\beta\in\Phi_+}(-c_\beta\cos\beta(T_1)+d_\beta\sin\beta(T_1))X_{-\beta}\right.\\
&+\left.\sum_{\beta\in\Phi_+}(c_\beta\sin\beta(T_1)+d_\beta\cos\beta(T_1))JX_{-\beta}\right] ~({\rm mod}\mathfrak k).
\end{align*}

Then we conclude that $X\in\mathfrak k$ if and only if 
$$\sinh\beta(T_2)\left(\begin{aligned}-&\cos\beta(T_1)&&\sin\beta(T_1)\\&\sin\beta(T_1)&&\cos\beta(T_1)\end{aligned}\right)\left(\begin{aligned}&c_\beta\\&d_\beta\end{aligned}\right)=\left(\begin{aligned}&0\\&0\end{aligned}\right),~\forall \beta\in\Phi_+.$$
Thus $X\in\mathfrak t$ if and only if
\begin{align}\label{lie-coe}
\sinh\beta(T_2)(c_\beta,d_\beta)=(0,0).
\end{align}

By the assumption that $\text{Ad}_tK'\subset K$, we see the Lie algebra $\text{Ad}_t\mathfrak t'\subset \mathfrak k$. If the decomposition of $\mathfrak k'$ according to \eqref{crtan-dec} has non-zero component on some $\mathfrak k_\alpha\oplus\mathfrak k_{-\alpha}$, then we can take $X\in\mathfrak k'$ such that $(c_\alpha,d_\alpha)\not=(0,0)$ in \eqref{lie-coe}. Hence $\alpha(T_2)=0$ for any such $\alpha$. We conclude that $e^{JT_2}\in C_G(K')$, which proves the lemma.

\end{proof}

\section{Appendix 2: Maximal torus in $\rm{Aut}(M)$}

In this Appendix, we will prove that when $\rm{Aut}(M)$ is reductive, the maximal torus of $\rm{Aut}(M)$ containing $T\times T$ is unique. We learned this result and its proof from Professor M. Brion \cite{Brion-letter}.

\begin{prop}\label{max-tor-unique}
Let $M$ be a polarized $G$-compactification such that ${\rm Aut}^0(M)$ is reductive. Then the maximal torus containing $T\times T$ (as a subgroup of ${\rm Aut}^0(M)$) is unique.
\end{prop}

Before the proof of Proposition \ref{max-tor-unique}, we shall first show the following lemma, which gives a characterization of torus:
\begin{lemm}\label{char-tor}
Let $H$ be a connected reductive algebraic group acting faithfully on a normal projective variety $M$.
Assume that every closed $H$-orbit is an  $H$-fixed point. Then $H$ is a torus.
\end{lemm}
\begin{proof}
By a result of Sumihiro \cite[Theorem 1]{Sumihiro}, there is an $H$-equivariant embedding of $M$ into some projective space $\mathbb{P}(V)$, where $V$ is a finite-dimensional representation of $H$. As $H$ is reductive, $V$ can be decomposed as direct sums of finitely many irreducible representations of $H$.

By assumption, every closed $H$-orbit in $M$ corresponds to an $H$-stable line in $V$, which is a $1$-dimensional representation of $H$. Let $V'$ be the direct sum of all the $1$-dimensional representations and $V''$ the direct sum of remaining ones. Then we decompose $V$ as sum $H$-invariant spaces of $$V=V'\oplus V''.$$

The resulting projection $V \to V'$ is $H$-equivariant,
and hance gives an $H$-equivariant rational map
$$\Pi:\mathbb{P}(V) \dashrightarrow \mathbb{P}(V').$$
This rational map is defined at any closed $H$-orbit in $M$ by construction. Since the set
\begin{align}\label{A0201}
M\cap\mathbb{P}(V'')=\emptyset
\end{align}
so that
$$f:=\Pi|_X:X\to\mathbb{P}(V')$$
is a morphism.
Suppose that \eqref{A0201} is not true, then $M\cap\mathbb{P}(V'')$ is an $H$-stable closed projective variety. Thus the $H$-orbit in it of lowest dimension must be closed. By our assumption it is an $H$-fixed point $x_0$. This implies that $V''$ contains a $1$-dimensional $H$-invariant space, contradicts to the construction of $V''$.

By \eqref{A0201} we in fact conclude that $f$ is a finite morphism.
By definition, the derived subgroup $[H, H]$ acts trivially on $V'$.  Since $[H, H]$ is connected and $f$ is finite, we see that $[H, H]$ acts trivially on $M$. Since $H$ acts faithfully there, $[H, H]$ is trivial and $H$ is a torus.
\end{proof}

\begin{proof}[Proof of Proposition \ref{max-tor-unique}]
By our assumption, ${\rm Aut}^0(M)$ is a connected
reductive algebraic group. Thus, the centralizer $H$ of $T\times T$ in ${\rm Aut}^0(M)$ is a connected reductive algebraic group as well.

It suffices to show that $H$ is a torus. Consider the fixed point set $M^{T \times T}$ in $M$. It is known that $M^{T \times T}$ is finite and
contained in the union of the closed $G \times G$-orbits (cf. \cite[Theorem 2.7]{AB2}).
Since $H$ commutes with $T\times T$, we conclude that it acts on $M^{T \times T}$. However, $M^{T \times T}$ is finite and $H$ is connected, we conclude that
$$M^H=M^{T \times T}.$$

On the other hand, let $Y$ be any closed $H$-orbit in $M$. By Borel's fixed point theorem,  $Y$ contains a fixed point $y_0$ of the subtorus  $T \times T$.
Thus $y_0$ is fixed by $H$, and hence $$Y=\{y_0\}$$
is a single point. By Lemma \ref{char-tor}, $H$ is a torus as desired.

\end{proof}


\begin{thebibliography}{10}
\bibitem{AB1}V. A. Alexeev and M. Brion, \textit{Stable reductive varieties \uppercase\expandafter{\romannumeral 1}: Affine varieties}, Invent. Math., \textbf{157} (2004), 227-274.

\bibitem{AB2} V. A. Alexeev,  and M. Brion, \textit{Stable reductive varieties \uppercase\expandafter{\romannumeral 2}: Projective case}, Adv. Math., \textbf{184} (2004), 382-408.

\bibitem{AK} V. A. Alexeev and  L. V. Katzarkov,  \textit{On K-stability of reductive varieties}, Geom. Funct. Anal., \textbf{15} (2005), 297-310.

\bibitem{AL}H. Azad and J. Loeb, \textit{Plurisubharmonic functions and K\"ahlerian metrics on complexification of symmetric spaces}, Indag. Math. (N.S.), \textbf{3} (1992), 365-375.

\bibitem{Berman-inve} R. Berman, \textit{K-stability of $\mathbb Q$-Fano varieties admitting K¡§ahler-Einstein metrics},  Invent. Math. \textbf{203} (2015), 973-1025.

\bibitem{BBEGZ} R. Berman, S. Boucksom, P. Eyssidieux, V. Guedj and A. Zeriahi, \textit{K\"ahler-Einstein metrics and the K\"ahler-Ricci flow on log Fano varieties}, arXiv:1111.7158v3, to appear in J. Reine Angew. Math.

\bibitem{Berman-Darvas-Lu} R. Berman, T. Darvas and C. Lu \textit{Regularity of weak minimizers of the K-energy and applications to properness and K-stability},  arXiv:1602.03114.

\bibitem{Brion-letter} M. Brion, \textit{Private communication}, 20, Aug. 2020.

\bibitem{Coman-Guedj-Sahin-Zeriahi} D. Coman, V. Guedj, S. Sahin and A. Zeriahi, \textit{Toric pluripotential theory}, arXiv:1804.03387.

\bibitem{DR} T. Darvas and Y. Rubinstein, \textit{Tian's properness conjectures and Finsler geometry of the space of K\"ahler metrics}, J. Amer. Math. Soc., \textbf{30} (2017), 347-387.

\bibitem{Del2} T. Delcroix.
\textit{K\"ahler-Einstein metrics on group compactifications}, Geom. Funct.
Anal., \textbf{27} (2017), 78-129.

\bibitem{Del3}T. Delcroix, \textit{K-Stability of Fano spherical varieties}, arXiv:1608.01852.

\bibitem{DT} W. Ding and G. Tian,   \textit{K\"ahler-Einstein metrics and the generalized Futaki invariants},  Invent. Math.,  \textbf{110} (1992), 315-335.

\bibitem{Di88} W. Ding, \textit{Remarks on the existence problem of positive K\"ahler-Einstein metrics},  Math. Ann.,  \textbf{282} (1988), 463-471.

\bibitem{Hisamoto}T. Hisamoto, \textit{Stability and coercivity for equivariant polarizations}, arXiv:1610.07998v3 (version of 8. Jul.  2019).

\bibitem{LZZ}Y. Li, B. Zhou and X. Zhu, \textit{K-energy on polarized compactifications of Lie groups}, J. Func. Analysis., \textbf{275} (2018), 1023-1072.

\bibitem{LZ}Y. Li and B. Zhou, \textit{Mabuchi metrics and properness of modified Ding functional},  Pacific J. Math., \textbf{302} (2019), 659-692.

\bibitem{LZ-Sasaki}Y. Li and X. Zhu, \textit{$G$-Sasaki manifolds and K-energy}, to appear on J. Geom. Anal.

\bibitem{LTZ2}Y. Li, G. Tian and X. Zhu, \textit{Singular K\"ahler-Einstein metrics on $\mathbb Q$-Fano compactifications of Lie groups}, arXiv:2001.11320.


\bibitem{Sumihiro} H. Sumihiro, \textit{Equivariant completion}, J. Math. Kyoto Univ. \textbf{14} (1974), 1-28.

\bibitem{Timashev-lecture-en}
D. A. Timashev, \textit{Equivariant embeddings of homogeneous spaces}, Surveys in geometry and number theory: reports on contemporary Russian mathematics, 226-278, London Math. Soc. Lecture Note Ser., \textbf{338}, Cambridge Univ. Press, Cambridge, 2007.

\bibitem{Timashev-book}
D. A. Timashev, \textit{Homogeneous spaces and equivariant embeddings}, Encyclopaedia of Mathematical Sciences, \textbf{138}. Invariant Theory and Algebraic Transformation Groups, 8. Springer, Heidelberg, 2011.


\bibitem {Zhu-survey}  X. Zhu, \textit{K\"ahler-Einstein metrics on toric manifolds and $G$-manifolds}, Geometric analysis, 545-585, Progress in Mathematics, \textbf{333}. Birkh\"auser Verlag, Basel, 2020.

\bibitem{Timashev-Sbo}
\font\fontWCA=wncyr8 {\fontWCA D. A. Timash\oe v}, \textit{\font\fontWCA=wncyr8 {\fontWCA \char'003kvivariantnye kompaktifikatsii reduktivnykh grupp}}, \font\fontWCA=wncyr8 {\fontWCA Matematicheski\ae\,  Sbornik}, \textbf{194} (2003), 119-146. \\
Eng.: D. A. Timash\"ev, \textit{Equivariant compactification of reductive groups}, Matematicheskij Sbornik, \textbf{194} (2003), 119-146.


\bibitem{Zhelobenko-Shtern}
\font\fontWCA=wncyr8 {\fontWCA D. P. Zhelobenko i A. I. Shtern}, \textit{\font\fontWCA=wncyr8 {\fontWCA Predstavleniya gruppy Li}}, \font\fontWCA=wncyr8 {\fontWCA Izdatel\~{}stvo Nauk, Moskva}, 1983.\\
Eng.: D. P. Zhelobenko and A. I. Shtern, \textit{Representations Lie groups}, Press ``Nauk", Moscow, 1983.


\end{thebibliography}
\end{document}